\documentclass[12pt,reqno]{amsart} 
\usepackage{geometry}
\usepackage{graphicx} 
\usepackage{amssymb}
\usepackage{epstopdf}
\usepackage{amsmath,amscd}
\usepackage{amsthm} 
\usepackage{url,verbatim}
\usepackage{mathtools}
\usepackage{xcolor}

\RequirePackage[colorlinks,citecolor=blue,urlcolor=blue]{hyperref}
\usepackage{breakurl} 
\theoremstyle{plain}
\newtheorem{theorem}{Theorem}
 
\newtheorem{lemma}[theorem]{Lemma}
\newtheorem{proposition}[theorem]{Proposition}

\newtheorem{remark}[theorem]{Remark}

\newtheorem{question}[theorem]{Question}

\newcounter{mycount}

\newenvironment{romlist}{\begin{list}{\rm(\roman{mycount})}%
{\usecounter{mycount}\labelwidth=1cm\itemsep 0pt}}{\end{list}}

\newenvironment{letlist}{\begin{list}{\rm(\alph{mycount})}%
{\usecounter{mycount}\labelwidth=1cm\itemsep 0pt}}{\end{list}}

\newcommand\XX{{\mathbb X}}

\newcommand\RR{{\mathbb R}} 
\newcommand\PP{{\mathbb P}} 
\newcommand\qq{\qquad} 
\newcommand\q{\quad}

\newcommand\Si{\Sigma}
 
\newcommand\om{\omega} 
\newcommand\Om{\Omega}

\newcommand\sB{{\mathcal B}} 

\newcommand\NN{{\mathbb N}}

\newcommand\resp{respectively}

\newcommand\oo{\infty}

\newcommand\sA{{\mathcal A}}

\renewcommand\o{\text{{\rm o}}}
\renewcommand\H{\red{\text{\rm H}}}
\newcommand\T{\blue{\text{\rm T}}}
\newcommand\A{\text{\rm A}}
\newcommand\B{\text{\rm B}}
\newcommand\Aa{\^ A}
\newcommand{\red}[1]{{\textcolor{red}{#1}}}
\newcommand{\blue}[1]{{\textcolor{blue}{#1}}}

\title[Alice and Bob on $\XX$: reversal, coupling, renewal]{Alice and Bob on $\XX$:\\ reversal, coupling, renewal}
\author{Geoffrey R.\ Grimmett}
\address{Centre for
Mathematical Sciences, Cambridge University, Wilberforce Road,
Cambridge CB3 0WB, UK} 
\email{grg@statslab.cam.ac.uk}

\date{submitted August 30, 2024, revised September 1, 2025} 
\keywords{Coin tosses, random game, central limit theorem}
\subjclass[2010]{60C05, 91A15}

\begin{document}
\begin{abstract}
A neat question involving coin flips surfaced on $\XX$, and generated an intensive
\lq storm' of `social mathematics'. In a sequence of flips of a fair coin,
Alice wins a point at each appearance of two consecutive heads, and Bob wins a point whenever a 
head is followed immediately by a tail. Who is more likely to win the game? 
The subsequent discussion illustrated conflicting intuitions, and concluded with the correct answer
(it is a close thing).
It is explained here why the context of the question is interesting and how it may be answered in a quantitative
manner using the probabilistic techniques of reversal, coupling, and renewal.
\end{abstract}

\maketitle

\section{The problem}

\begin{question}\label{qn:1}
A fair coin is tossed $n$ times. Alice scores one point  at each appearance of two
consecutive heads, and Bob scores one point each time a head is followed immediately by a tail. 
The winner is the player who accumulates more points.
Who is the more likely to win?
\end{question}

Here are some intuitive arguments in order of decreasing naivet\'e.
\begin{letlist}
\item In any pair of consecutive coin tosses, Alice wins a point with probability $\frac14$,
and Bob wins a point with the same probability. Therefore, each has a mean total
score of $\frac14(n-1)$. Since these means are equal, the game is fair and each player has the same probability of winning.
\item Alice's points tend to arrive in clusters, whereas Bob's are isolated. 
That favours Alice, so she is more likely to win.
\item When Alice wins a point, she has an increased chance of winning again. 
Therefore, when she wins, she tends to do so by a wider margin than does Bob.
However, her \emph{mean total}
is the same as Bob's. It follows that Bob has a greater \emph{probability} of winning.
\end{letlist}

This question (with $n=100$) was posed by Daniel Litt on his $\XX$ feed \cite{Litt}
on 16 March 2024. 
At the current time of writing, his post has attracted 1.2M views. 
First responders tended to favour Alice above Bob on the grounds of
argument (b), and a vote was reported as placing
Alice (26.3\%) over Bob (10.2\%), with 42.8\%\ of the 51,588 voters supporting equality. 
Later simulations appeared to show that Bob has a slight advantage.

ChatGPT has changed its position on the question over the intervening months. Initially it favoured 
Alice on the grounds given in (b) above. At the time of writing it has veered towards (a),
with the conclusion that
``The game is fair to both players in terms of expected outcomes.''
Fair enough, but not very helpful. There remains hope for mathematicians.\footnote{\emph{Note on revision}:
This seems to be the view of ChatGPT also, which has now discovered the reference \cite{Seg}.}

More complete answers began to surface on Litt's $\XX$ feed fairly soon after the original post,
and
suggestions  were made for rigorous proof. It is not easy to convey the details 
and to check the correctness of a mathematical
argument within the
confined format of  $\XX$, and  hence this note. We present three conclusions.

\begin{theorem}\label{thm:0}
Consider $n$ flips of a fair coin.
\begin{letlist}
\item Bob is (strictly) more likely than Alice  to win when $n \ge 3$.
\item $\PP(\text{\rm Bob wins}) - \PP(\text{\rm Alice wins}) \sim c/\sqrt n$ as $n\to\oo$, where 
$$
c= \frac1{2\sqrt \pi}\approx 0.282.
$$
\item
The probability of a tie is asymptotically $1/\sqrt {\pi n}$.
\end{letlist}
\end{theorem}

Explicit representations
for the probabilities in (b) and (c) are given in
equations \eqref{eq:3n}--\eqref{eq:13n}. 
Parts (b) and (c) imply that
\begin{equation}\label{eq:51}
\frac12 - \PP(\text{B wins}) \sim \frac{1}{4\sqrt{\pi n}}, \qq
\frac12 - \PP(\text{A wins}) \sim \frac{3}{4\sqrt{\pi n}}.
\end{equation}

It is immediate that Alice and Bob are equally likely to win when $n=1,2$. By considering
the eight possible outcomes when $n=3$, we have
$\PP(\text{Bob wins}) - \PP(\text{Alice wins}) = \frac18$ in this case. 

The methods of proof may be summarized as reversal, coupling, and renewal,
and the proof of Theorem \ref{thm:0} illustrates these standard techniques. Reversing 
a random sequence is a long-established activity which has contributed
enormously to probability  and especially to the theory of random walks
(see, for example, \cite[Sect.\ 3.10]{PRP}). The coupling  approach enables
us to study the \lq pathwise\rq\ relationship between winning sequences for Alice and 
for Bob, rather
than by simply calculating probabilities (see, for example, \cite[Sect.\ 4.12]{PRP}). 
Renewal theory is the study of random processes 
that renew themselves at random times (see for example, \cite[Chap.\ 10]{PRP}).

We make two intuitive remarks about the problem
(see also the next section). Let $S_n$ denote Bob's score  minus Alice's score after $n$
flips. 
\begin{letlist}
\item
Certainly, $S_n$ has mean $0$, but that does not imply that its distribution is symmetric, and it 
is precisely such  asymmetry that skews
the game towards Bob. One measure of asymmetry is termed \lq skewness' (denoted \lq skw'), and is defined in terms
of the so-called third central moment (see \cite[p.\ 163]{PRP}). The sign of the skewness (positive or negative) 
is an indicator of the more likely winner.
\item
The  term $\sqrt n$ of Theorem \ref{thm:0}
originates in the fact that $S_n$ is the sum of independent summands.
The number $S_n$ is random but it has order $n$. Intuition based on the central limit theorem suggests that $S_n$ has
an approximately Gaussian distribution with mean $0$ and variance of order $n$. In particular, one
may guess that $\PP(S_n=0)$ has order $1/\sqrt n$, in line with the so-called local central limit theorem
(see \cite[p.\ 219]{PRP}). As further evidence, the skewness of the sum of $n$ independent copies of a random variable $X$
is skw$(X)/\sqrt n$.
\end{letlist}

\section{Remarks on related problems}\label{sec:rem}

One may phrase the problem addressed here as  \lq$\H\T$ vs $\H\H$\rq, 
and we shall write \lq$\H\T > \H\H$\rq\ in paraphrase of Theorem \ref{thm:0}(a).
This implies that $\T\H > \H\H$ for the following reason. 
Take a sequence $\om$ of heads/tails and reverse it in time to obtain $\rho(\om)$ --- so that, for example,
the sequence $\om = \H\T\H\H\H\T$ (sometimes abbreviated in the natural way to $\H\T\H^3\T$)
becomes the reversed sequence $\rho(\om) = \T\H\H\H\T\H$ (abbreviated to $\T\H^3\T\H$).
Then $\om$ and $\rho(\om)$ have the same probability distribution, and the same count of consecutive 
head pairs $\H\H$. However, every $\H\T$ in $\om$ becomes $\T\H$ in $\rho(\om)$.
Such reversal is a technique used heavily in the current paper. Another simple technique is the interchange of 
heads and tails, whereby one deduces immediately that $\T\H > \T\T$.

Further questions spring easily to mind, involving sequences of length greater than $2$, and indeed 
\emph{sets} of sequences.
There is a general method for determining the likely winner for sufficiently large $n$, and it can be applied 
to more general problems than $\H\H$ vs $\H\T$. 
This is not the subject of the current article, but we outline it here for completeness.
The aggregate score process $S$ evolves as the number of flips increases, and it does so in the manner of
a stochastic process with finite-range dependence (that is, there exists $r < \oo$ such that
the increments in $S$ are independent when they are separated in time by $r$ or more). 
Such processes have many of the properties of random walk, and in particular they satisfy a central limit theorem (CLT),
in which the rate of convergence to the Gaussian distribution is determined (in the main) by the so-called
third central moment.  The last can be calculated explicitly in specific cases, and this yields the 
identity of the likely winner when $n$ is large. Some technicalities must be overcome to complete this argument,
and the nice arguments of Feller \cite[Sect.\ XVI.4]{F2} may be useful in this regard.
Mention is made of the preprint of Basdevant et al.\ \cite{bas}.
A full account of an approach of this type is expected in \cite{JNS}.

This CLT argument yields a general solution to such problems, but only for sufficiently large values of $n$.
Since it relies on asymptotics and error estimates, 
it cannot be expected to answer such questions for all values of $n$. 
In addition, it does not reveal detailed aspects of structure such as those described 
in the current work for the case $\H\T$ vs $\H\H$.
In a sense, Litt's question was a lucky shot (though fortune favours the well prepared) --- it was simple,
captivating, and has a beautiful answer. There is a limited number 
of other cases that are at least partially susceptible to the methods of the current article.

As mathematicians we insist on unambiguous definitions of the objects of our study. 
One may capture pretty well
all of probability theory by defining it
as the theory of a countably infinite sequence of tosses of a fair coin.  Perhaps not 
everything worth knowing is yet known about
this primeval experiment.

The target of the current note is to obtain Theorem \ref{thm:0} using exact probabilistic methods
to illuminate structure.
The associated literature is slightly complex, owing to the nature and speed of 
communication of $\XX$.
The idea of reversal has arisen independently in certain contributions on $\XX$ to Litt's post (see \cite{SR}, for example, for 
some nice ideas). 
Mention is made of \cite{EZ}, which uses computational tools to explore the quantities in Theorem \ref{thm:0},
 and \cite{Seg}, which  obtains
the theorem by applying analytic tools to generating functions.

\section{Notation and basic observations}\label{sec:not}

Here is some notation. Abbreviate Alice to A, and Bob to B; write $\H$ for a head and $\T$ for a tail.
Let 
$\NN$ denote the natural numbers, and let $\Om_\oo=\{\H,\T\}^\NN$ 
be the
set of all sequences $\om=(\om_1,\om_2,\dots)$
each element of which is either $\H$ or $\T$.  
We shall normally express such sequences as $\om=\om_1\om_2\cdots$.
Similarly, let $\Om_n:= \{\H,\T\}^n$,
the set of all sequences of $n$ coin tosses. 

Let $\om\in\Om_\oo$.
Player A scores $-1$ at each appearance of 
$\H\H$; player B scores $1$ at each appearance of $\H\T$.
The score $S_k(\om)$ is the aggregate score after 
$k$ steps of $\om$, that is, $S_k(\om)$ is the number
of appearances of $\H\T$ minus the number of appearances of $\H\H$ in 
the subsequence $\om_1\om_2\cdots\om_k$.

There follows an outline of the method to be followed here, with an illustration in Figure \ref{fig:1}. 
First, one defines epochs of renewal, 
at which the scoring process restarts. Neither player scores until the first appearance of $\H$. 
\begin{romlist}
\item If the following flip is $\T$ then Bob enters a winning period (of some length $l_\B$), 
which lasts until the next time that the aggregate score is $0$;
this must happen at an appearance of $\H\H$.  

\item If the following flip is $\H$ then Alice enters a winning period (of some length $l_\A$),
which lasts until the next time that the aggregate score is $0$;
this must happen at an appearance of $\H\T$. 
\end{romlist}
The process then restarts (subject to 
certain details to be made specific). After $n$ coin flips, Alice is winning if she is then in a winning period, 
and similarly for Bob.  Ties occur between winning periods.
We will see that the representative periods $l_\A$, $l_\B$ are such that
$l_\A$ is (stochastically) smaller than $l_\B$ (that is to say, $\PP(l_\A > l)\le \PP(l_\B>l)$ for all $l$), 
and Theorem \ref{thm:0}(a) will follow.

\begin{figure}
\begin{equation*}
\begin{array}{r c c  c c c c c c c  c c c c  c c c c c c c}
k: &1&2&3&4&5&6&7&8&9&10&11&12&13&14&15&16&17&18&19&20\\
\om_k:&\T&\H&\T&\H&\T&\H&\H&\H&\H&\T&\T&\H&\T&\T&\H&\T&\H&\H&\H&\T\\
S_k:  &0&&1&&2&&1&0&-1&0&&&1&&&2&&1&0&-1\\
&&&\rlap{$\xleftrightarrow[\textstyle l_\B]{\phantom{xxxxxxxxxxxxxxxxxx}}$}&&&&&&\rlap{$\xleftrightarrow[\textstyle l_\A]{\phantom{xxxx}}$} &&&&\rlap{$\xleftrightarrow{\phantom{xxxxxxxxxxxxxxxxxxxxxxxxxx}}{}$}\\
\end{array}
\end{equation*}
\caption{A sequence $\om=\T\H\T\H\T\H^4\T^2\H\T^2\H\T\H^3\T$
of 20 coin flips. The scores $S_k$ are stated only when they change. 
Bob enters a winning period from epoch $3$ to 
epoch $8$, then Alice from $9$ to $10$, and then Bob from $13$ to $19$. }\label{fig:1}
\end{figure}

The stochastic domination is proved by
displaying a concrete coupling of $l_\A$ and $l_\B$. 
The above explanation is given in more detail in Section \ref{sec:3}.

The asymptotic part (b) of Theorem \ref{thm:0} is of course connected to the local
central limit theorem (see, for example, \cite[p.\ 219]{PRP}). It suffices to use the earliest such theorem ever proved, namely the 1733 theorem of de Moivre
\cite{deM1,PdMA} (see also \cite[p.\ 243--254]{deM} and \cite[Thm 1.1]{SW}), 
although one may equally use the (de Moivre--)Stirling formula alone 
(see the historical note \cite{KP}).

We have assumed implicitly in the above that the lengths $l_\A$, $l_\B$ 
are finite (almost surely), and this
requires proof. It follows from the next lemma.

Let $\PP$ be the probability measure on $\Om_\oo$ under which the coin tosses 
are independent random variables,
each being $\H$ (\resp, $\T$) with probability $\frac12$. 
An event $E$ is said to occur \emph{almost surely} if
$\PP(E)=1$.

\begin{lemma}\label{lem:2}
The number of times $r$ at which the aggregate score satisfies
$S_r(\om)=0$ is infinite almost surely.
\end{lemma}

\begin{remark}
Many of the arguments of this note are valid in the more general setting where heads occurs with some probability $p\in(0,1)$. However, the conclusion of Lemma \ref{lem:2} is 
false when $p\ne \frac12$, and indeed
$\PP(\text{\rm A wins})$ approaches $1$ if $p>\frac12$, and approaches $0$ if $p<\frac12$,
in the limit as the number of coin flips grows to $\oo$.
\end{remark}

\begin{proof}[Proof of Lemma \ref{lem:2}]
There are several ways to prove this, of which we feature the neat argument of \cite{SR}.
Let $Z_1,Z_2,\dots$ be independent random variables taking the values $\pm 1$ each with probability $\frac12$.
The partial sums $T_n := Z_1+Z_2+\dots+ Z_n$ are said to form a \emph{simple random walk} (SRW).
It is a famous theorem that
$$
\PP(T_n=0 \text{ for some $n\ge 1$})=1
$$ 
(see \cite[Cor.\ 5.3.4]{PRP}), and this is usually expressed by saying that the value $0$ is \lq recurrent\rq.
Having returned once, the walk will (with probability $1$) return again, and so on. Therefore,
\begin{equation}\label{eq:2+}
\PP(T_n=0 \text{ for infinitely many $n$})=1.
\end{equation}

The score process $(S_r)$ is not a SRW, but instead a \lq time-changed' SRW. We view it as follows.
First wait until the first $\H$. Exactly one of the two following events then occurs, each having probability $\frac12$,
\begin{equation*}
\begin{aligned}
\text{(tail):} &\text{ the next flip is $\T$ and $S$ increases by $1$,}\\
\text{(head):} &\text{ the next flip is $\H$ and $S$ decreases by $1$.}\\
\end{aligned}
\end{equation*}
In case (tail), we wait for the next $\H$; in either case, the process is then iterated.  
Let $Z_i$ be the $i$th change in value of the 
score process. By the above, the $Z_i$ are the steps of a SRW, albeit delayed in time. Every return to $0$ of
this SRW marks a return to $0$ of the score, and the claim follows by \eqref{eq:2+}.
\end{proof}

\section{Proof of Theorem \ref{thm:0}($\mathrm{a}$)}\label{sec:3}

A finite sequence $\om=\om_1\cdots\om_k$ is called 
\begin{align*}
&\text{a B-\emph{excursion} if it starts 
$\H\T$ and ends $\H\H$, and satisfies:}\\
&\hskip3cm \text{$S_l(\om)>0$ for $1<l<k$, and $S_k(\om)=0$,}\\
&\text{an A-\emph{excursion} if it starts 
$\H\H$ and ends $\H\T$, and satisfies:}\\
&\hskip3cm \text{$S_l(\om)<0$ for $1<l<k$, and $S_k(\om)=0$,}\\
&\text{an \Aa-\emph{excursion} if it comprises an A-excursion followed by a (possibly empty) }\\
&\hskip3cm\text{run of tails and then a single head. }
\end{align*}
We write $\sB$ for the set of B-excursions, and $\sB_k$ for those of length $k$,
with similar notation for A-excursions and \Aa-excursions (using
the notation $\sA$ and $\hat \sA$, \resp).

\Aa-excursions are introduced for the following reason. 
Players A and B can score only at a coin flip that follows an appearance of $\H$. 
It will be convenient that excursions finish with a head.
Since an A-excursion ends with a tail, we simply extend it
to include any following run of tails, followed by the subsequent head; this extension contains neither $\H\H$ nor $\H\T$ and hence the score of the excursion is unchanged.

\begin{remark}\label{rem:2}
A \emph{headrun} (\resp, \emph{tailrun}) is a maximal consecutive sequence of heads
(\resp, tails). Any finite sequence $\om$ comprises alternating headruns and tailruns.
Alice's score equals the number of heads minus the number of headruns. Bob's score
equals the number of tailruns having a preceding head. Let $h$ be the total number
of heads, and $r$ the total number of runs (each comprising either heads or tails).
The aggregate score $S(\om)$
equals $r-h$ if $\om_1=\H$ (\resp, $r-1-h$ if $\om_1=\T$). 
\end{remark}

We define sequence-reversal next. 
Let $\Phi=\bigcup_{k=1}^\oo \Om_k$ be the set of all non-empty finite sequences,
and let $\Phi^{\H}$ be the subset of $\Phi$ containing
all finite sequences that begin $\H$.
For $\om\in\Phi$, we write $S(\om)$ for 
the aggregate score of $\om$. The score function is additive in the sense that
\begin{equation}\label{eq:2n}
S(\om_1\cdots\om_{m+n}) = S(\om_1\cdots\om_{m})+S(\om_m\cdots\om_{m+n}),
\qq m,n \ge 1.
\end{equation}
For a finite sequence
$\om= \om_1\om_2\dots \om_k\in\Phi$,
define its \emph{reversal} $\rho(\om)$ by $\rho(\om)= \om_k\om_{k-1}\dots \om_1$.

 \begin{lemma}\label{lem:1}
 \mbox{\hfil}
 \begin{letlist}
 \item If $\om=\om_1\cdots\om_k\in\Phi^\H$ ends in $\H$ (so that $\om_1=\om_k=\H$), 
 then $S(\om)=S(\rho(\om))$.
 \item Let $k\ge 2$.  The mapping $\rho$ is a bijection between the sets $\sB_k$ and
 $\hat \sA_k$.
 \end{letlist}
 \end{lemma}
 
\begin{proof}
(a) Let $\om\in\Phi^\H$ end in $\H$, so that $\rho(\om)\in \Phi^\H$.
By Remark \ref{rem:2}, $S(\om)$ equals the number of runs minus the number of heads. These counts are invariant under $\rho$.
 
 (b) A B-excursion $b=b_1b_2\cdots b_k\in \sB_k$ has the form $b=\H\T^r\H\cdots \H\H$ for some $r\ge 1$, 
 whence $\rho(b)=b_kb_{k-1}\cdots b_1$
 has the form $\H\H\cdots\H\T^r\H$. 
By part (a), $S(\rho(b))=S(b)=0$, which implies, after removal of the final $\T^{r-1}\H$, that
\begin{equation}\label{eq:21n}
S(b_kb_{k-1}\cdots b_{r+1})=0. 
\end{equation}

 We claim that 
 \begin{equation}\label{eq:20n}
 S(b_kb_{k-1}\cdots b_j) <0 \qq\text{for  } j \in [r+2,k-1].
 \end{equation}
Suppose, on the contrary, that $j\in[r+2,k-1]$ is such that 
\begin{equation}\label{eq:200}
S(b_kb_{k-1}\cdots b_j) =0,
\end{equation}
and pick $j$ largest with this property. Since $b_kb_{k-1}\cdots b_j$ starts with $\H\H$, it must end with
$\H\T$. We now extend it to the next appearance in $\rho(b)$ of $\H$, thus obtaining 
$b_kb_{k-1}\cdots b_j \T^a\H$ (= $b_k b_{k-1}\cdots b_{j-a-1}$)
for some $a\ge 0$.
By \eqref{eq:200}, 
\begin{equation}
S(b_k b_{k-1}\cdots b_{j-a-1})= S(b_kb_{k-1}\cdots b_j \T^a\H)=0.
\end{equation}
 By part (a), we have 
 $$
 S(b_{j-a-1}\cdots b_{k-1}b_k) = S(b_kb_{k-1}\cdots b_{j-a-1})=0.
 $$ 
 By the additivity of $S$, \eqref{eq:2n}, 
 $$
 S(b_1b_2\cdots b_{j-a-1}) = S(b_1b_2\cdots b_k) - S(b_{j-a-1}\cdots b_{k-1}b_k) = 0,
 $$
 in contradiction of the assumption that $b$ is a B-excursion, and \eqref{eq:20n} follows.  
 
We illustrate the above argument with an example. Suppose $k=8$ and $b=\H\T^2\H\T\H^3$, so that $\rho(b)=
\H^3\T\H\T^2\H$. Note that $b_kb_{k-1}\cdots b_{r+1}=\H^3\T\H\T$ in agreement with \eqref{eq:21n}.
Equation \eqref{eq:20n} amounts to the assertion that $S(\H^3\T\H), S(\H^3\T), S(\H^3), S(\H^2) < 0$.

We return to the proof of the lemma.
By \eqref{eq:21n}--\eqref{eq:20n}, $b_kb_{k-1}\cdots b_{r+1}$ is an A-excursion, and hence
$\rho(b)$ is an \Aa-excursion.
Therefore, $\rho$ is  an injection.
By the same argument applied to $\rho^{-1}$, we have that $\rho$ is a surjection.
\end{proof}
 
\begin{figure}
\begin{alignat*}{2}
&\rlap{$\phantom{\xleftrightarrow[\textstyle M]{\phantom{xxxxxxxxxxxxx}}}
\xleftrightarrow[\phantom{xxxxxxxiixxxxxxxxxxxxixxxxxxx}]{\textstyle N}$}
\\
\noalign{\vskip-12pt}
&\begin{array}{cccc}\T&\T&\T&\H\end{array}  
\Bigl\{
&&\begin{array}{c c  c c c c c }
\T&\T&\H&\T&\H&\H&\H\\
\H&\H&\T&\H&\T&\T&\H
\end{array}
\Bigr\}\cdots\\
\noalign{\vskip-10pt}
&\rlap{\,$\xleftrightarrow[\textstyle M]{\phantom{xxxxxxxxxxxxx}} 
\xleftrightarrow[\textstyle Q]{\phantom{xxxxxxxyxxxxxxxxxxxxx}}
\xleftrightarrow[\textstyle M']{\phantom{xxxxxx}}$}
\end{alignat*}

\caption{An illustration of the notation $M$, $N$, $Q$. Within the braces, 
the upper sequence is a B-excursion, and the lower is an \Aa-excursion.}\label{fig:2}
\end{figure}
 
 We make three observations, in amplification of the remarks of the last section.
Let $\om\in\Om_\oo$, and observe the initial evolution of scores, as illustrated in Figure \ref{fig:2}. 

\begin{letlist}
\item There is no score until the first $\H$ appears. Let $M$ be the position of the first
$\H$, and note that 
\begin{equation}\label{eq:3+}
\PP(M=m) = (\tfrac12)^m, \qq m\ge 1.
\end{equation}
That is, there is an initial 
sequence of tails of some random length $M-1$ ($\ge 0$), followed by a head 
(so that $\om=\T^{M-1}\H\cdots$). 
Following this head, there is equal probability of $\H$ or $\T$.
\item If the next toss is $\T$ (which is to say that $\om_{M+1}=\T$), then Bob scores one point.
He remains in the lead until the next epoch, $M+N$ say, at which the aggregate score equals $0$.
It must be the case that $\om_{M+N-1}\om_{M+N} = \H\H$,
and thus the sequence $\om_M\cdots\om_{M+N}$ is a B-excursion. 
The process restarts from $\H$ at time $M+N$.

\item  Suppose the next toss is $\H$
(that is, $\om_{M+1}=\H$). The situation is slightly more complicated in this case.
At time $M+1$, Alice leads by $1$, and she continues in the lead until
the next epoch, $M+Q$ say, when the aggregate score equals $0$, and moreover 
$\om_{M+Q-1}\om_{M+Q} = \H\T$. Thus the sequence $\om_1\cdots\om_{M+Q}$ is an 
A-excursion.
Whereas in (b) the process restarts from the state 
$\om_{M+N}=\H$,
this time we have $\om_{M+Q}=\T$, and we wait a random period of time for the next $\H$. 
As in (a), there is a run of tails before
the next head, which is to say that the process restarts from the head at epoch $M+Q+M'$ where
$M'$ is an independent copy of $M$. The sequence $\om_M\cdots\om_{M+N+M'}$ is an
\Aa-excursion. 
\end{letlist}

We shall adjoin certain sequences by placing them in tandem, and we write either $\psi_1\psi_2$
or $\psi_1\cdot\psi_2$
for $\psi_1$ followed by $\psi_2$. When scoring $\psi_2$, viewed
as a subsequence of $\psi_1\psi_2$, one  
takes account of the final character of $\psi_1$, which will typically be $\H$ in the cases studied here. 
In order to do the necessary book-keeping we introduce the following notation: 
for a sequence $\psi$ beginning with $\H$,
we write $\H^{-1}\psi$ for the sequence obtained from $\psi$
by removing its initial $\H$.

We construct next a random sequence of excursions drawn from the set $\sA$ of A-excursions.
Consider a random sequence of flips beginning $\H\H$, and let $\tau$ be the initial A-excursion of this sequence;
we term $\tau$ a \emph{random A-excursion}; a similar definition holds for a random B-excursion.
Let $(\tau_i: i\ge 1)$ be a 
sequence of independent copies of $\tau$ (that is, they are independent and each has the same distribution as $\tau$).
Let  $Q_i+1$ be the length of $\tau_i$ (so that $\H^{-1}\tau_i$
has length $Q_i$).
The $\tau_i$ are almost surely finite, by Lemma \ref{lem:2}.
Let $(M_i:i\ge 0)$ be independent  copies of $M$ (see \eqref{eq:3+}), also independent of the $Q_i$.
Let $N_i=Q_i+M_i$, and let $L_i$ be the sequence $\T^{M_i-1}\H$.

We have that $\tau_i L_i$ is a random element of $\hat\sA$ of length $Q_i+M_i+1$.
By Lemma \ref{lem:1}, $\rho(\tau_i L_i)$ is a random element of $\sB$ of length $Q_i+M_i+1$.
We define
\begin{equation}\label{eq:5+}
\begin{aligned}
&\text{the \Aa-excursion $\alpha_i:=\tau_i L_i$},\\
&\text{the B-excursion $\beta_i:=\rho(\tau_i L_i)$}.
\end{aligned}
\end{equation}
Thus $\beta_i$ is simply a reversal of $\alpha_i$, and this fact will provide a coupling of 
\Aa-excursions and B-excursions which is length-conserving (in that $\alpha_i$ and
$\beta_i$ have the same length).

Both $\alpha_i$ and $\beta_i$ start and end with $\H$; when placing them in tandem
we shall strip the initial $\H$.
Another way of expressing (a)--(c) is as follows. 
\begin{romlist}
\item  A random sequence of coin tosses begins with $L_0$.

\item This is followed by 
\begin{equation*}
\gamma_1 := \begin{cases}
\H^{-1}\alpha_1 &\text{with probability $\frac12$},\\
\H^{-1}\beta_1 &\text{with probability $\frac12$}.
\end{cases}
\end{equation*}

\item Let $k\ge 2$, and suppose $\alpha_i$, $\beta_i$ have been constructed for
$i=1,2,\dots,k$. We then let
  \begin{equation}
\gamma_{k+1} := \begin{cases}
\H^{-1}\alpha_{k+1} &\text{with probability $\frac12$},\\
\H^{-1}\beta_{k+1} &\text{with probability $\frac12$}.
\end{cases}
\label{eq:8n}
\end{equation}
\end{romlist}
Note that the length of $\gamma_i$ is $Q_i+M_i$.

\begin{lemma}\label{lem:3}
The sequence $X=L_0\gamma_1\gamma_2\cdots$ is an independent 
sequence of fair coin tosses.
\end{lemma}

This does not require proof beyond the above remarks.  
The lemma provides a representation of a random sequence in terms of an initial tailrun, followed by
independent copies of $\gamma_1$ in tandem. Each $\gamma_i$ occupies a time-slot, and within this
slot there appears a sequence of coin-tosses; according to the flip of another fair coin, we either retain this
sequence or we reverse it (see \eqref{eq:8n}).   
(Some minor details concerning initial and final heads are omitted from this overview).

The sequence $X$ may be expressed as follows:
\begin{equation}\label{eq:4+}
X = L_0 \cdot\left\{
\begin{matrix} \H^{-1}\alpha_1 = \H^{-1}\tau_1L_1\\ \H^{-1} \beta_1 =\H^{-1}\rho(\tau_1L_1)\end{matrix} 
\right\} \cdot
\left\{
\begin{matrix}\H^{-1}\alpha_2 = \H^{-1}\tau_2L_2\\ \H^{-1} \beta_2 =\H^{-1} \rho(\tau_2 L_2) \end{matrix}
\right\}\cdots.
\end{equation}
There are two possibilities within each matched pair of braces, and one chooses between them at random. 
More precisely, let $Z_1,Z_2,\dots$ be 
independent random variables taking values in $\{a,b\}$, with each value chosen with probability $\frac12$. 
We pick the upper
element of the $i$th matched pair if $Z_i=a$ and the lower element if $Z_i=b$.

Let $X=(X_k:k\ge 1)$ be as in Lemma \ref{lem:3}, so that $X$ may be expressed in the form of \eqref{eq:4+}.
We shall explore the differences between the upper and lower elements of
any given pair of braces there. Consider the $i$th matched pair of braces in \eqref{eq:4+}.
Let $f_i$ be the first index in this subsequence of $X$, and let $s_i$ be the last. Finally, let
$m_i= f_i+|\H^{-1}\tau_i|-1$ where $|\om|$ denotes the length of a sequence $\om$; that is, $m_i$ is 
the index of the last element of $\H^{-1}\tau_i$ when $Z_i=a$. In illustration of these three indices, we
assume that $Z_i=a$ and express the $i$th matched pair in the form:
\begin{equation}\label{eq:6+}
\left\{\H^{-1}\tau_iL_i\right\} = (\underset{\downarrow f_i}{t_2}t_3\dots \underset{\downarrow m_i}{t_c})
\cdot(l_1l_2\dots \underset{\downarrow s_i}{l_d})
\end{equation}
where $\tau_i= \H t_2\dots t_c$ and $L_i=l_1l_2\dots l_d$. Then $f_i$ is the index of $t_2$ in $X$, $m_i$
is the index of $t_c$, and $s_i$ is the index of $l_d$ (\lq$f$' for \emph{first} point, \lq$m$' for 
\emph{middle}, and \lq$s$' for \emph{second}).
 Evidently, $m_i < s_i$. 

Here is a concrete example.
If $\tau_i = \H^3\T\H^2\T\H\T$ and $L_i = \T^3\H$, then the $i$th pair of braces contains
$\H^{-1}\tau_iL_i = \H^2\T\H^2\T\H\T \cdot \T^3\H$, so that 
$m_i=f_i+ 7$ and $s_i=m_i + 4$, where $f_i$ is the index in $X$
of the first element of this brace.

\begin{proof}[Proof of Theorem \ref{thm:0}(a)]
Let $n\ge 3$, and find the unique $I=I_n$ such that $n$ lies in the $I$th brace pair of \eqref{eq:4+} (write $I=0$ if
no such pair exists). Note that $f_I\le n\le s_I$ for $I\ge 1$. 
If $n$ lies inside a B-excursion (\resp, A-excursion), then Bob (\resp, Alice) is currently winning.
The third possibility is that the $I$th brace contains an \Aa-excursion and $n$ falls at
or after the end of the A-excursion therein.
More precisely, the score at epoch $n$ satisfies
\begin{equation}
\label{eq:40}
\begin{aligned}
\text{$S_n=0$}&\text{\q if\q}&&I=0,\\
\text{$S_n>0$}&\text{\q if\q}&&I\ge 1,\, Z_I=b, \text{ and } n< s_I\,\\
\text{$S_n<0$}&\text{\q if\q}&&I\ge 1,\,  Z_I=a, \text{ and } n < m_I,\\
\text{$S_n=0$}&\text{\q if\q}&&\text{none of the above conditions hold}.
\end{aligned}
\end{equation}
Recalling \eqref{eq:4+}, it follows that
\begin{align}\nonumber
\PP(S_n>0)- \PP(S_n<0) &= \tfrac12\bigl[\PP(I\ge 1,\,n< s_I)-\PP(I\ge 1,\, n< m_I)\bigr]\\
&=\tfrac12\PP(I\ge 1,\, m_I\le n < s_I)\ge 0.
\label{eq:3n}
\end{align}
The \emph{strict} positivity of the last probability (with $n\ge 3$)
follows by consideration of sequences beginning $\H^2\T^{m}\H\cdots$
with $m\ge n$,
for which $M_0=1$, $M_1=m$, and $\gamma_1=\H\T^m\H$. Thus, $\tau_1=\H^2\T$,
and the interval $\{m_1,m_1+1,\dots,s_1-1\} = \{3 ,4,\dots,m+2\}$ contains
the value $n$.
\end{proof}

The probability of a tie is derived similarly to \eqref{eq:3n} as
\begin{equation}\label{eq:13n}
\PP(S_n=0) =\PP(I=0) + \PP(I\ge 1,\, n= s_I)
+ \tfrac12\PP(I\ge 1,\, m_I \le n < s_I).
\end{equation}
From \eqref{eq:3n}--\eqref{eq:13n} (as in \eqref{eq:51})
one may obtain representations for the probabilities $\PP(S_n>0)$ and $\PP(S_n<0)$.

\section{Proof of Theorem \ref{thm:0}($\mathrm{b,\,c}$)}\label{sec:3}

Returning to the representation \eqref{eq:4+} for $X$, we take the lower element in each matched pair of braces, 
and reorganize the
terms slightly to obtain the sequence
\begin{equation}\label{eq:7+}
Y:= \{L_0 \H^{-1} \tau_1\}\cdot \{L_1 \H^{-1}\tau_2\}\cdots.
\end{equation}
We emphasize that $Y$ and $X$ are very different, but nevertheless  the sequence $Y$
enables us to locate a set of renewal points that will aid the analysis of $X$, namely the following.
Let $R = \{m_0,m_1,m_2,\dots\}$ where $m_0=0$ and the $m_i$ for $i\ge 1$ are as marked in \eqref{eq:6+}; 
in other words,
$R$ is the set of indices of the rightmost elements of the braces in \eqref{eq:7+}. 
Since $(m_{i+1}-m_i: i \ge 0)$  are independent and identically distributed, $R$ is indeed a renewal process.
Such a process has the \emph{renewal property} that, conditional on $\{m\in R\}$ the process subsequent to time $m$
has the same distribution as the process starting at time $0$. 
Let $\pi_m:=\PP(m \in R)$.

Let $m\le n$, and recall the notation $Z_i$, $I=I_n$ from the last section.
By the renewal property, conditional on the event $\{m\in R,\, Z_I=a\}$,
the sequence $\psi_{m,n}:=X_{m+1}X_{m+2}\cdots X_n$
is a sequence of independent coin flips. For fixed $n$, the events 
$$
E_{m,n}=\{m\in R,\, Z_I=a,\,\psi_{m,n}=\T^{n-m}\},
\qq m=1,2,\dots,n,
$$
are disjoint with union 
$$
\bigcup_{m\le n} E_{m,n} = \{I\ge 1,\, Z_I=a,\,m_I\le n < s_I\}.
$$
By \eqref{eq:3n} and the above,
\begin{align}\nonumber
\PP(S_n>0)- \PP(S_n<0) &= \PP(I\ge 1,\, Z_I=a, \, m_I\le n < s_I)\nonumber\\
&=\sum_{m=1}^n \PP(m\in R, \psi_{m,n} = \T^{n-m}\mid Z_I=a)\PP(Z_I=a)\nonumber\\
&=\frac12\sum_{m=1}^n \PP(m\in R,\,\psi_{m,n}=\T^{n-m})
=\sum_{k=0}^{n-1} (\tfrac12)^{k+1}\pi_{n-k}.
\label{eq:4n}
\end{align}
Theorem \ref{thm:0}(b) follows once the following has been proved.

\begin{proposition}\label{prop:1}
We have that $\pi_m \sim c/\sqrt m$ as $m\to\oo$, where $c=1/(2\sqrt \pi)\approx 0.282$.
\end{proposition}

The proof uses the de Moivre local central limit theorem (as in \cite{SW}),
though this amounts only to sustained use of the (de Moivre--)Stirling formula, \cite{KP}.

An outline of the proof of Theorem \ref{thm:0}(c) is included after that of the proposition.

\begin{proof}This proof is based on a combinatorial calculation associated with a certain subset of the renewal process 
$R=(m_i: i \ge 0)$, namely the subset $R_X=\{m_i: i\ge 1,\, Z_i=a\}$.
It is immediate that $R_X \subseteq R$. Moreover, by the independence of  $(m_i)$ and $(Z_j)$,
\begin{equation}\label{eq:31n}
\PP(m\in R_X) = \tfrac12 \PP(m\in R), \qq m\ge 1.
\end{equation} 

By Lemma \ref{lem:3}, the event $\{m\in R_X\}$ is the set of all $\om$ with length $m$ 
that satisfy $\om_{m-1}\om_m=\H\T$ and $S(\om)=0$.
Let $h$ be the number of heads in such $\om$, and $r$ the number of runs (of either heads or tails). 
If $r$ is even, (\resp, odd), then such sequences begin $\H$ (\resp, $\T$), and they invariably
finish with a tailrun of exactly one tail. By Remark \ref{rem:2}, since the aggregate score is $0$,
we have $h=r$ if $r$ is even, and $h=r-1$ if $r$ is odd.

We use the fact that the number of $p$-partitions of the integer $q$ is 
$\binom {q-1}{p-1}$.
Let $m\ge 5$, say.
By partitioning the heads and tails into $r$ runs subject to the above, and interleaving headruns and tailruns,
we obtain that
$$
\vert\{m\in R_X\}\vert= \Si_1+\Si_2
$$
where
\begin{align*} 
\Si_1&= \sum_{r \text{ odd}} 
\binom{m-h-2}{\frac12(r+1)-2} \binom{h-1}{\frac12(r-1)-1}
=\sum_{r \text{ odd}} 
\binom{m-r-1}{\frac12(r-3)} \binom{r-2}{\frac12(r-3)} ,\\
\Si_2&= \sum_{r\text{ even}}
\binom{m-h-2}{\frac12 r-2} \binom{h-1}{\frac12 r-1}
= \sum_{r\text{ even}}
\binom{m-r-2}{\frac12 r-2} \binom{r-1}{\frac12 r-1}.
\end{align*}
Set $r=2s+1$ (\resp, $r=2s$) in $\Si_1$ (\resp, $\Si_2$) and add to obtain
\begin{align}\nonumber
\vert\{m\in R_X\}\vert
 &= 
\sum_{s=1}^\oo \left[\binom{m-2s-2}{s-1}+ \binom{m-2s-2}{s-2}\right]
\binom{2s-1}{s-1}\\
&=\sum_{s=1}^\oo \binom{m-2s-1}{s-1}\binom{2s-1}{s-1}.\label{eq:9n}
\end{align}
We could now expand the binomial coefficients and use Stirling's formula directly. 
These calculations have been done already by de Moivre \cite{deM1,SW}, and we choose to refer to his results, as follows.

By \eqref{eq:9n},
\begin{align}
\PP(m\in R_X) &= \left(\frac12\right)^m \vert\{m\in R_X\}\vert\nonumber\\
&= \frac14 \sum_{s=1}^\oo \PP(T_{m-2s-1}= s-1)\PP(T_{2s-1}=s -1),\label{eq:11n}
\end{align}
where $T_k$ has the binomial distribution with paramters $k$ and $\frac12$. 
By Stirling's formula (or \cite[Thm 1.1]{SW}),
\begin{equation}\label{eq:10n}
\PP\left(T_{2s-1}=s -1\right) \sim \frac 1{\sqrt{\pi s}}\q\text{as } s\to\oo.
\end{equation}

We may occasionally use real numbers in the following where integers are expected.
The term $\PP(T_{m-2s-1}=s -1)$ in \eqref{eq:11n} is a maximum when 
$m-2s-1=2(s-1)$, which is to say that $s=\frac14 (m+1)$. Let $\gamma\in(\frac12,\frac23)$.
We may restrict the summation in \eqref{eq:11n} to values of $s$ 
satisfying $\vert s-\frac14 m\vert < m^\gamma$. To see this, note that
\begin{align*}
\sum_{s\ge  \frac14m+ m^\gamma} \PP(T_{m-2s-1} = s-1)
&\le \sum_{s\ge  \frac14m+ m^\gamma} \PP(T_{m/2} \ge s-1),
\end{align*}
which tends to $0$ as $m\to\oo$ by the moderate-deviation theorem of Cram\'er \cite{Cram}
(or, for a more modern treatment, see Feller \cite[p.\ 549]{F2}\footnote{Feller's theorem
may be stated as follows. Let $(Z_i)$ be a sequence of independent identically distributed random variables with zero 
mean and unit variance, and with finite moment generating function in some neighborhood of $0$, 
and write $S_n=Z_1+Z_2+\cdot+ Z_n$. The central limit theorem
states that, for $x\in\RR$, $\PP(S_n>x\sqrt n)/[1-\Phi(x)] \to 1$ as $n\to \oo$, where $\Phi$ 
denotes the standard Gaussian
distribution function. The same limit is valid as $x,n\to\oo$ in such a way that $x=\o(n^{1/6})$.}).
We have used the fact that $T_k$ is 
stochastically increasing in $k$ in that, for $t\ge 0$, $\PP(T_k > t)$ is non-decreasing in $k$. 
A similar argument applies for $s\le \frac14 m - m^\gamma$.

Suppose $\vert s - \frac14 m\vert < m^\gamma$.
By \cite[Thm 1.2]{SW}, there is an absolute constant $C$ such that
$$
\left\vert \PP(T_{m-2s-1}=s -1)- \sqrt{\frac 2{\pi (m-2s-1)}} 
\exp\left(-\frac{(4s-m-1)^2}{2(m - 2s-1)}\right)\right\vert 
\le \frac C {m^{2/3}}.
$$
It follows that
$$
\sum_{\vert s - \frac14 m\vert < m^\gamma} 
\PP(T_{m-2s-1}=s -1)
$$
deviates from 
$$
\Psi_m:=\sum_{\vert s - \frac14 m\vert < m^\gamma}\sqrt{\frac 2{\pi (m-2s-1)}}
\exp\left(-\frac{(4s-m-1)^2}{2(m - 2s-1)}\right)
$$
by at most $C(2m^\gamma+1)/m^{2/3}$, which tends to $0$.
Express $\Psi_m$ as an integral, make the change of variable 
$$
\beta = \frac{4s-m-1}{\sqrt{m-2s-1}}
$$
and let $m\to\oo$ to obtain, by the dominated convergence theorem, that
$$
\Psi_m \to  \frac12\int_{-\oo}^\oo \frac1{\sqrt{2\pi}} e^{-\frac12 \beta^2} \, d\beta = \frac12.
$$
Therefore, the terms $\PP(T_{m-2s-1}=s -1)$ in \eqref{eq:11n} are (asymptotically)
concentrated in the interval $\frac14m\pm m^\gamma$, with total weight $\frac12$. Hence,
by \eqref{eq:31n} and \eqref{eq:11n}--\eqref{eq:10n}, 
$$
\pi_m\sim \frac2{8\sqrt{\pi m/4}} =   \frac1{2\sqrt{\pi m}},
$$
as claimed.
\end{proof}

\begin{proof}[Outline proof of Theorem \ref{thm:0}(c)]
One may deduce part (c) from \eqref{eq:13n} by adapting the argument leading to \eqref{eq:4n}.
Alternatively, one may perform a direct calculation as above, and there follows a sketch of this. 
Let $V_m$ be the set of vectors $\om\in\Om_m$ such that $S(\om)=0$.
Elements $\om\in V_m$ may be expressed as interleaved headruns and tailruns, with the counts
of heads and tailruns being balanced by the condition $S(\om)=0$ (see Remark \ref{rem:2}). 
Such $\om$ may start with either $\H$ or $\T$, and each case leads to two terms as in the first line of
\eqref{eq:9n}. Thus, $\vert V_m\vert$ is the sum of four terms, each being the product of two 
binomial coefficients. The analysis continues as in the above proof.
\end{proof}

\section*{Acknowledgments} 
GRG thanks Daniel Litt for a sweet question, and Charles McGinnes for drawing his attention to it.
He thanks Richard Weber, David Stirzaker, and Svante Janson for useful correspondence, and is grateful to the referees
for their valuable reports which have led to corrections and improvements.

\providecommand{\bysame}{\leavevmode\hbox to3em{\hrulefill}\thinspace}
\providecommand{\MR}{\relax\ifhmode\unskip\space\fi MR }
\providecommand{\MRhref}[2]{%
  \href{http://www.ams.org/mathscinet-getitem?mr=#1}{#2}
}
\providecommand{\href}[2]{#2}


\begin{thebibliography}{10}

\bibitem{bas}
A.-L. Basdevant, O.~H\'enard, \'E. Maurel-Segala, and A.~Singh, \emph{On cases
  where {L}itt's game is fair},  (2024),
  \url{https://arxiv.org/abs/2406.20049}.

\bibitem{Cram}
H.~Cram\'er, \emph{Sur un nouveau th\'eor\`eme-limite de la th\'eorie des
  probabilit\'es}, Actualit\'es Scientifiques et Industrielles \textbf{736}
  (1938), 2--23, transl.\ H.\ Touchette,
  \url{https://arxiv.org/abs/1802.05988}.

\bibitem{deM1}
A.~de~Moivre, \emph{{Approximatio ad Summam Terminorum Binomii $\overline{a +
  b}^n$ in Seriem expansi}},  (1733).

\bibitem{deM}
\bysame, \emph{{The Doctrine of Chances}}, 3rd ed., 1756,
  \url{https://archive.org/details/doctrineofchance00moiv/page/n5/mode/2up}.

\bibitem{EZ}
S.~B. Ekhad and D.~Zeilberger, \emph{{How to answer questions of the type: if
  you toss a coin $n$ times, how likely is HH to show up more than HT?}},
  (2024), \url{https://arxiv.org/abs/2405.13561}.

\bibitem{F2}
W.~Feller, \emph{{An Introduction to Probability Theory and its Applications.
  {V}ol. {II}}}, 2nd ed., John Wiley \& Sons., New York, 1971.

\bibitem{PRP}
G.~R. Grimmett and D.~R. Stirzaker, \emph{{Probability and Random Processes}},
  4th ed., Oxford University Press, 2020.

\bibitem{JNS}
S.~Janson, M.~Nica, and S.~Segert, \emph{{The generalized Alice HH vs Bob HT
  problem}},  (2025), \url{https://arxiv.org/abs/2503.19035}.

\bibitem{Litt}
D.~Litt,  (2024), \url{https://x.com/littmath/status/1769044719034647001}.

\bibitem{KP}
K.~Pearson, \emph{Historical note on the origin of the normal curve of errors},
  Biometrika \textbf{16} (1924), 402--404.

\bibitem{PdMA}
K.~Pearson, A.~de~Moivre, and R.~C. Archibald, \emph{{A rare pamphlet of Moivre
  and some of his discoveries}}, Isis \textbf{8} (1926), 671--683,
  \url{https://www.journals.uchicago.edu/doi/epdfplus/10.1086/358439}.

\bibitem{SR}
S.~Ramesh,  (2024),
  \url{https://x.com/RadishHarmers/status/1770217475960885661}.

\bibitem{Seg}
S.~Segert, \emph{{A proof that HT is more likely to outnumber HH than vice
  versa in a sequence of $n$ coin flips}},  (2024),
  \url{https://arxiv.org/abs/2405.16660}.

\bibitem{SW}
Z.~Szewczak and M.~Weber, \emph{Classical and almost sure local limit
  theorems}, Dissertationes Math. \textbf{589} (2023), 97 pp.

\end{thebibliography}
\end{document}